\documentclass[11pt,leqno]{article}
\usepackage{geometry}
\newgeometry{vmargin={28mm}, hmargin={35mm,35mm}}   

\usepackage{xcolor}

\usepackage[hidelinks]{hyperref}

\usepackage{amsmath,amsthm,amsfonts,amssymb,latexsym,amscd,enumerate}
\usepackage{slashed}
\usepackage{mathtools}

\usepackage{palatino}
\usepackage[mathcal]{euler}

\usepackage{xy}
\xyoption{all}

\numberwithin{equation}{section}
\swapnumbers

\newtheorem{theorem}{Theorem}[section]
\newtheorem*{theorem*}{Theorem}

\newtheorem{proposition}[theorem]{Proposition}
\newtheorem*{proposition*}{Proposition}

\theoremstyle{definition}
\newtheorem{definition}[theorem]{Definition}

\newtheorem{observation}[theorem]{Observation}

\numberwithin{equation}{section}

\begin{document}

\title{Enlargeable Foliations and the Monodromy Groupoid: Infinite Covers}

\author{Guangxiang Su and Zelin Yi}

\date{}

\maketitle

\abstract{In this paper, we prove that the foliated Rosenberg index of a possibly noncompactly enlargeable, spin foliation is nonzero. It generalizes our previous result. The difficulty brought by the noncompactness is reflected in the infinite dimensionality of some vector bundles which, fortunately, can be reduced to finite dimensional vector bundles by the idea of relative index theorem and $KK$-equivalence between the $C^\ast$-algebra of compact operators and $\mathbb{C}$.}

\section{Introduction}
Enlargeability \cite{GromovLawson80} is an important notion in studying which manifold admit positive scalar curvature metric. There are several notions of enlargeability for foliation in the literatures. 
Recall the following notion of enlargeability which is, in some sense, in between that of \cite[Definition~0.1]{Zhang20} and \cite[Definition~1.5]{BenameurHeitsch19}.
\begin{definition}\label{def-enlargeable-foliation}
	A foliation $(M,F)$ is compactly enlargeable if there is $C>0$ such that for any $\varepsilon> 0 $, there is a compact covering $\widetilde{M}_\varepsilon$ of $M$ and a smooth map
	\[
	f_\varepsilon: \widetilde{M}_\varepsilon \to S^n
	\]
	with
	\begin{itemize}
		\item 
		$|f_{\varepsilon,\ast} X| \leq \varepsilon |X|$ for all $X\in C^\infty(\widetilde{M}_\varepsilon,\widetilde{F}_\varepsilon)$, where $\widetilde{F}_\varepsilon$ is the lifting of $F$ to $\widetilde{M}_\varepsilon$;
		\item
		$|f_{\varepsilon,\ast} X| \leq C\cdot |X|$ for all $X\in C^\infty(\widetilde{M}_\varepsilon, T\widetilde{M}_\varepsilon)$;
		\item
		$f_\varepsilon$ has nonzero degree;
		\item
		$f_\varepsilon$ is constant outside some compact subset $K_\varepsilon\subset \widetilde{M}_\varepsilon$.
	\end{itemize}
A foliation $(M,F)$ is enlargeable if the above condition holds with possibly non-compact coverings $\widetilde{M}_\varepsilon$.
\end{definition}

We define, in \cite{SuYi23}, a foliated version of Rosenberg index $[\alpha]\in K_\ast (C^\ast G_M)$ and prove
\begin{theorem}
	If $(M,F)$ is a compactly enlargeable foliation with $F$ spin, then the foliated Rosenberg index $[\alpha]\in K_\ast(C^\ast G_M)$ is nonzero.
\end{theorem}
The main result of this paper is the following generalization.

\begin{theorem}\label{thm-main-result}
	If $(M,F)$ is an enlargeable foliation with $F$ spin, then the foliated Rosenberg index $[\alpha]\in K_\ast(C^\ast G_M)$ is nonzero.
\end{theorem}

\subsection{Previous approach}
Let us first recall the idea of \cite{SuYi23} as follow:  Let $E\to S^n$ be a vector bundle whose top degree Chern class is not zero while all other Chern classes are zero. The pullback $f_\varepsilon^\ast E\to \widetilde{M}_\varepsilon$ can be "wrapped up" (see \cite{HankeSchick06}) to a finite dimensional vector bundle $E_\varepsilon\to M$. As $\varepsilon$ range over $1,1/2,\cdots, 1/n,\cdots$, we get a sequence of leafwise increasingly flat vector bundles $\{E_{1/i}\}_{i\in \mathbb{N}}$ over $M$ whose Chern classes vanish except the top degree one. Denote by $d_{1/i}$ the dimension of $E_{1/i}$.

\begin{definition}\label{def-algebra-A-Q}
	Let $A\subset \prod M_{d_{1/i}}(\mathbb{C})$ be the $C^\ast$-algebra of bounded sequence, $A^\prime \subset A$ be the subalgebra consisting of sequences that converge to zero. In other word, $A^\prime$ is the closure of 
	$
	{\bigoplus \prod M_{d_{1/i}}(\mathbb{C})}\subset A.
	$
	Let $Q=A/A^\prime$.
	
\end{definition}

The sequence of vector bundles $\{E_{1/i}\}$ can be assembled into a Hilbert $A$-module bundle $\mathbb{E}\to M$ and the increasingly-flatness is reflected in the fact that its leafwise curvature take value in $A^\prime$.  As a consequence, the quotient vector bundle $\mathbb{E}/A^\prime$ by the action of $A^\prime$ is genuine leafwise flat Hilbert $Q$-module bundle. According to \cite[Corollary 6.10]{SuYi23}, it determines a homomorphism $K(C^\ast G_M)\to K_0(C^\ast (G_M,Q))$ which send the foliated Rosenberg index to the longitudinal index of the twisted leafwise Dirac $D_{\mathbb{E}/A^\prime}$. It then suffices to show that the index of  $D_{\mathbb{E}/A^\prime-Q\times M}$, which is defined as $\operatorname{Ind}(D_{\mathbb{E}/A^\prime})-\operatorname{Ind}(D_{Q\times M})$, is nonzero. The argument is summarize in the following diagram:
\begin{equation}\label{eq-diagram-chase}
\xymatrix{
	K_0(C^\ast (G_M, A^\prime))\ar[d] \ar[r]& \bigoplus K_0(C^\ast G_M) \ar[d]  \ar^{\mu}[r] &\bigoplus\mathbb{C}\ar[d] \\
	\operatorname{Ind}(D_{\mathbb{E}-A\times M})\in K_0(C^\ast (G_M, A))\ar[r]\ar[d] & \prod K_0(C^\ast G_M) \ar^{\mu}[r] &\prod \mathbb{C},\\
	\operatorname{Ind}(D_{\mathbb{E}/A^\prime-Q\times M})\in K_0(C^\ast (G_M, Q)) &  &
}
\end{equation}
where the first column is exact, the horizontal arrows from the first column to the second column are induced by sending $A$ and $A^\prime$ to their components and the horizontal arrows from the second column to the third column is induced by Connes' transversal fundamental class. The twisted leafwise Dirac operator $D_{\mathbb{E}/A^\prime-Q\times M}$ can be lifted to $D_{\mathbb{E}-A\times M}$. It suffices to show that the longitudinal index of $D_{\mathbb{E}-A\times M}$ does not come from the image of the top-left arrow. Indeed, it can be checked that the longitudinal index of $D_{\mathbb{E}-A\times M}$ is sent, by the horizontal sequence, to
\[
 \left(\langle\widehat{A}(F)\operatorname{ch}(E_{1/i}-\mathbb{C}^{d_i}), [M]\rangle \right)_{i\in \mathbb{N}},
\]
which, according to the assumption on Chern classes, has infinitely many nonzero terms.

\subsection{Difficulty with noncompactness}

The main difficulty with non-compactly enlargeable foliation is that the covering spaces $\widetilde{M}_\varepsilon$ are non-compact so that the bundles $E_\varepsilon\to M$ are infinite dimensional. To get some useful information out of $E_\varepsilon$'s, Hanke and Schick \cite{HankeSchick07} organize these bundles into some Hilbert module bundles in the following novel way:

Let $G$ and $H_\varepsilon$ be the fundamental groups for $M$ and $\widetilde{M}_\varepsilon$ respectively. Then each fiber of the covering $\widetilde{M}_\varepsilon\to M$ can be parameterized by $G/H_\varepsilon$. We fix such a parameterization. The $C^\ast$-algebras $C_T,C_S$ and $C_{S,T}$ are as in \cite{HankeSchick07} except that we shall add a supscript $\varepsilon$ to indicate its dependence on $\varepsilon$.

\begin{definition}
	Let $C_T^\varepsilon\subset \mathcal{B}(\ell^2(G/H_\varepsilon)\otimes \mathbb{C}^d)$ be the $C^\ast$-algebra generated by $G/H_\varepsilon$ families of $M_d(\mathbb{C})$ all but finitely many are zero. Let $C^\varepsilon_S\subset \mathcal{B}(\ell^2(G/H_\varepsilon)\otimes \mathbb{C}^d)$ be the $C^\ast$-algebra generated by permutations of $G/H_\varepsilon$ and let $C_{S,T}^\varepsilon$ be the $C^\ast$-algebra generated by $C^\varepsilon_S$ and $C^\varepsilon_T$. Notice here $C_T^\varepsilon$ is isomorphism to the $C^\ast$-algebra of compact operators and is a two-sided ideal in $C_{S,T}^\varepsilon$.
\end{definition}

Fix $\varepsilon$ and let $\pi_\varepsilon: \widetilde{M}_\varepsilon\to M$ be the covering map, let $\{U_\alpha\}$ be a finite open cover of $M$ such that $f^\ast_\varepsilon E$ is trivial on $\pi_\varepsilon^{-1}(U_\alpha)$ for all $\alpha$. Fix these trivializations
\begin{equation}\label{eq-trivializations-of-pullback}
f^\ast_\varepsilon E|_{\pi^{-1}_\varepsilon(U_\alpha)} \xrightarrow{\varphi_\alpha} \pi^{-1}_\varepsilon(U_\alpha)\times \mathbb{C}^d.
\end{equation}
The transition functions 
\begin{equation}\label{eq-transition-function}
\pi^{-1}_\varepsilon (U_\alpha\cap U_\beta)\times \mathbb{C}^d\xrightarrow{\varphi_{\alpha\beta}} \pi^{-1}_\varepsilon (U_\alpha\cap U_\beta)\times \mathbb{C}^d
\end{equation}
can be viewed as a map from $U_\alpha\cap U_\beta$ to $C^\varepsilon_{S,T}$, which, used as new set of transition functions, build the Hilbert $C_{S,T}^\varepsilon$-module bundle $E_\varepsilon\to M$.

Apart from Hilbert module bundle structure, one more ingredient is need to tackle the noncompactness of $\widetilde{M}_\varepsilon$, that is \textbf{relative index}. Let $R(C^\varepsilon_{S,T})$ be the $C^\ast$-algebra defined by $R(C^\varepsilon_{S,T})=\{(c_1,c_2)\in C^\varepsilon_{S,T}\times C^\varepsilon_{S,T} \mid c_1-c_2\in C_T^\varepsilon\}$. Inspired by Roe's approach to relative index \cite{Roe91}, Hanke and Schick organize the virtual bundle $E_\varepsilon-C^\varepsilon_{S,T}$ into a Hilbert $R(C^\varepsilon_{S,T})$-module bundle. Then the index of the  twisted Dirac operator $D_{E_\varepsilon-C^\varepsilon_{S,T}}$ belongs to $K_0(R(C^\varepsilon_{S,T}))$. 

We shall recycle the letters $A, A^\prime$ and $Q$ from Definition~\ref{def-algebra-A-Q} to make the following new definition.
\begin{definition}\label{def-another-a-q}
	Let $A\subset \prod_i R(C^{1/i}_{S,T})$ be the $C^\ast$-algebra of bounded sequences, $A^\prime\subset A$ be the subalgebra consists of sequences that converge to zero and $Q=A/A^\prime$.
\end{definition}
Following the same pattern, we may assemble the sequence of virtual bundles $\{E_{1/i}-C^{1/i}_{S,T}\}$ into a Hilbert $A$-module virtual bundle $\mathbb{E}-A$. Its quotient $\mathbb{E}/A^\prime-Q$ is a virtual flat Hilbert $Q$-module bundle which determines a homomorphism $\pi_1(M)\to Q$. The advantage of $R(C^\varepsilon_{S,T})$ is that its $K$-theory splits out a $\mathbb{Z}$-component which can be computed by a generalization of Mischenko-Fomenko index theorem \cite{Schick05} plus a trace calculation. An analogous diagram as \eqref{eq-diagram-chase} can be considered:
\begin{equation}\label{eq-diagram-chase-unfoliated}
	\xymatrix{
		K_0(A^\prime)\ar[d] \ar[r]& \bigoplus K_0(R(C_{S,T}^{1/i})) \ar[d]  \ar[r] &\bigoplus\mathbb{Z}\ar[d] \\
		\operatorname{Ind}(D_{\mathbb{E}-A\times M})\in K_0(A)\ar[r]\ar[d] & \prod K_0(R(C_{S,T}^{1/i})) \ar[r] &\prod \mathbb{Z},\\
		\operatorname{Ind}(D_{\mathbb{E}/A^\prime-Q\times M})\in K_0(Q) &  &
	}
\end{equation}
where again the horizontal map from the first column to the second column is given by mapping $A$ and $A^\prime$ to their components. The difference is that the horizontal arrows from the second column to the third column is given by the $\mathbb{Z}$-components in the $K$-theory of $R(C^\varepsilon_{S,T})$. As a consequence the image of the middle horizontal sequence is
\[
\left(\langle \widehat{A}(\widetilde{M}_{1/i})\operatorname{ch}(f^\ast_{1/i} E-\mathbb{C}^d), [\widetilde{M}_{1/i}]\rangle\right)_i,
\]
which has infinitely many nonzero terms.

\subsection{Reduction to compact case}

To circumvent the usage of \cite{Schick05} in the case of foliation, we make the following observation.
\begin{observation}\label{observation}
	As $C^\varepsilon_T\subset C^\varepsilon_{S,T}$ is an ideal, the $C^\ast$-algebra $C^\varepsilon_{S,T}$ can be mapped to the multiplier algebra of $C^\varepsilon_T$
	\begin{equation}\label{eq-CST-to-multiplier-algebra}
	C^\varepsilon_{S,T}\to M(C^\varepsilon_T),
	\end{equation}
	the compositions of the set of transition functions \eqref{eq-transition-function} and the map \eqref{eq-CST-to-multiplier-algebra} can be taken as a new set of transition functions. Together with trivializations $U_\alpha\times C^\varepsilon_T$, it builds  a Hilbert $C^\varepsilon_T$-module bundle $E_\varepsilon$. Notice that here we recycle the notation $E_\varepsilon$. From here on, the notation $E_\varepsilon$ will be reserved for this Hilbert $C_T^\varepsilon$-module bundle.
\end{observation} 
Now the difficulty of non-compactness is reflected in the fact that $C^\varepsilon_T\cong \mathcal{K}$ is not unital so that the $KK$-theory element construction \cite[Corollary 6.10]{SuYi23} can not be applied directly. 

Again the missing ingredient is provided by the relative index theorem. The virtual bundle 
\begin{equation}\label{eq-difference-k-bundle}
E_\varepsilon - C^\varepsilon_T
\end{equation}
can be organized into an element in the group $K_0(C(M)\otimes C^\varepsilon_{T})$. Under the light of $KK$-equivalence between $\mathcal{K}$ and $\mathbb{C}$ the difference bundle \eqref{eq-difference-k-bundle} can be reduced to a finite dimensional virtual bundle $E^{0}_\varepsilon - \mathbb{C}^{d_{\varepsilon}}$ with $d_\varepsilon=\operatorname{dim}(E^0_\varepsilon)$. The advantage of $E^{0}_\varepsilon$ over $E_\varepsilon$ is that, apart from being asymptotic flat as $\varepsilon\to 0$, it is finite dimensional vector bundle so that \cite[Corollary 6.10]{SuYi23} can be applied. Exactly the same diagram as \eqref{eq-diagram-chase} can be applied, and the image of the middle horizontal sequence is given by 
\begin{equation}\label{eq-chern-character-integral}
	\left(\langle \widehat{A}(F)\operatorname{ch}(E^0_{1/i}-\mathbb{C}^{d_i}), [M]\rangle\right)_{i\in \mathbb{N}} = \left(\langle \widehat{A}(\widetilde{F}_{1/i})\operatorname{ch}(f^\ast_{1/i} E-\mathbb{C}^{d}), [\widetilde{M}_{1/i}]\rangle\right)_{i\in \mathbb{N}}
\end{equation}
which is a sequence that has infinitely many nonzero terms. This leads to Theorem~\ref{thm-main-result}.

\section{Construction and properties of $E^0_\varepsilon$}

This section is organized as follows: in subsection~\ref{sec-difference-bundles} the virtual bundle \eqref{eq-difference-k-bundle} is discussed,  in subsection~\ref{sec-kk-equivalence} the equivalent finite dimensional virtual bundle is defined, in subsection~\ref{sec-asymptotically-flat} the asymptotically flatness is proved  and finally in subsection~\ref{sec-chern-character}, the equation \eqref{eq-chern-character-integral} is proved.

\subsection{Difference of $\mathcal{K}$-bundles}\label{sec-difference-bundles}

Due to the fourth bullet point in the Definition~\ref{def-enlargeable-foliation}, the pullback $f_\varepsilon^\ast E\to \widetilde{M}_\varepsilon$ and the trivial bundle $\mathbb{C}^d\to \widetilde{M}_\varepsilon$ are isomorphic outside the compact subset $K_\varepsilon\subset \widetilde{M}_\varepsilon$. Since $\widetilde{M}_\varepsilon$ is locally compact, there is an open neighborhood $K_\varepsilon\subset K_\varepsilon^\prime$ whose closure is compact. Let $\theta_\varepsilon: f_\varepsilon^\ast E\to \mathbb{C}^d$ be the unitary outside $K^\prime_\varepsilon$ which, according to the Tietze extension theorem, admits an extension to $\widetilde{M}_\varepsilon$. We shall use the same notation to denote the extension.

\[
\xymatrix{
f_\varepsilon^\ast E \ar[r]^{\theta_\varepsilon}  \ar[d]& \mathbb{C}^d \ar[d]&&E_\varepsilon \ar[r]^{\theta^\mathcal{K}_\varepsilon}  \ar[d]& C^\varepsilon_{T} \ar[d]\\
\widetilde{M}_\varepsilon \ar[r]^{\operatorname{id}} &\widetilde{M}_\varepsilon&&M \ar[r]^{\operatorname{id}} &M
}
\]

\begin{proposition}\label{prop-theta-local}
There is a bundle map 
\[
\theta_\varepsilon^\mathcal{K}: E_\varepsilon\to C^\varepsilon_T
\]
such that the restriction $\theta_\varepsilon^\mathcal{K}|_{U_\alpha}$ is given by left multiplication of element in $C(U_\alpha, C_T^{\varepsilon,+})$ for all $\alpha$. Here we denote by $C^{\varepsilon,+}_T$ the one-point unitlization.
\end{proposition}

\begin{proof}
Under the trivializations \eqref{eq-trivializations-of-pullback}, the bundle map $\theta_\varepsilon: f^\ast_\varepsilon E \to \mathbb{C}^d$ is given by 
\[
\theta_\varepsilon|_{\pi^{-1}_\varepsilon(U_\alpha)}\circ \varphi^{-1}_\alpha,
\]
which can be taken as a $G/H_\varepsilon$ family of $M(\mathbb{C}^d)$ all but finitely many are $1$. In other word, it can be taken as a map from $U_\alpha$ to $C^{\varepsilon,+}_T$. Define $\theta_\varepsilon^\mathcal{K}$ to be the left multiplication with $\theta_\varepsilon|_{\pi^{-1}_\varepsilon(U_\alpha)}\circ \varphi^{-1}_\alpha$ over $U_\alpha$ for all $\alpha$. It is clear that this definition of $\theta_\varepsilon^\mathcal{K}$ is invariant under the transition functions of $E_\varepsilon$.
\end{proof}

\begin{proposition}
	Denote by $C(M, E_\varepsilon\oplus C^\varepsilon_{T})$ the graded Hilbert module over $C(M)\otimes C^\varepsilon_{T}$ whose even part is given by $E_\varepsilon$ and the odd part is given by the trivial bundle $C^\varepsilon_{T}$. Then the triple
\begin{equation}\label{eq-Kasparov-module-of-k-bundles}
	\left(C(M, E_\varepsilon\oplus C^\varepsilon_{T}),\begin{bmatrix}
		1&0\\
		0&1
	\end{bmatrix},
\begin{bmatrix}
	0 & \theta_\varepsilon^{\mathcal{K},\ast} \\
	\theta_\varepsilon^\mathcal{K} & 0
\end{bmatrix}
\right)
\end{equation}
is a Kasparov module in $KK(\mathbb{C}, C(M)\otimes C^\varepsilon_{T})$.
\end{proposition} 

\begin{proof}
	 It suffices to check that $\begin{bmatrix}
	 	0 & \theta_\varepsilon^{\mathcal{K},\ast} \\
	 	\theta_\varepsilon^\mathcal{K} & 0
	 \end{bmatrix}^2-1$ is a compact operator. Indeed, according to Proposition~\ref{prop-theta-local}, $\theta^\mathcal{K}_\varepsilon \theta^{\mathcal{K},\ast}_\varepsilon-1$ and $\theta^{\mathcal{K},\ast}_\varepsilon \theta^\mathcal{K}_\varepsilon-1$ are given by left multiplication with elements in $C^\infty(M, C^\varepsilon_T)$ and $C(M,\mathcal{K}(E_\varepsilon))$ respectively.
\end{proof}

\subsection{KK-equivalence}\label{sec-kk-equivalence}

Let $\mathcal{H}$ be the standard separable Hilbert space. Analogous to the observation~\ref{observation}, the same set of transition functions \eqref{eq-transition-function} together with the trivializations $U_\alpha\times \mathcal{H}$ build a bundle of Hilbert spaces $\mathcal{H}_\varepsilon$, and parallel to Proposition~\ref{prop-theta-local}, there is a bundle map $\theta_\varepsilon^\mathcal{H}: \mathcal{H}_\varepsilon\to \mathcal{H}\times M$ such that $\theta_\varepsilon^{\mathcal{H}}\theta_\varepsilon^{\mathcal{H},\ast}-1$ and $\theta_\varepsilon^{\mathcal{H},\ast}\theta_\varepsilon^{\mathcal{H}}-1$ are given by compact operators.

The $KK$-equivalence between $\mathcal{K}$ and $\mathbb{C}$ is implemented by the elements $x=(\mathcal{H}, 1, 0) \in KK(\mathcal{K}, \mathbb{C})$ and $y=(\mathcal{K},p_1,0 )\in KK(\mathbb{C}, \mathcal{K})$ where $p_1\in \mathcal{K}$ is some rank one projection. Under the Kasparov product $KK(\mathbb{C}, C(M)\otimes \mathcal{K})\otimes KK(\mathcal{K}, \mathbb{C})\to K_0(C(M))$ the Kasparov module \eqref{eq-Kasparov-module-of-k-bundles} becomes
\begin{equation}\label{eq-Kasparov-module-of-bundles}
	\left(C(M, \mathcal{H}_\varepsilon\oplus \mathcal{H}), \begin{bmatrix}
		1&0\\
		0&1
	\end{bmatrix}, 
\begin{bmatrix}
	0 & \theta_\varepsilon^{\mathcal{H},\ast}\\
	\theta_\varepsilon^{\mathcal{H}} & 0
\end{bmatrix}
\right).
\end{equation}
We recall a trick used in \cite{AtiyahSinger71IndexIV}, to find an equivalent finite dimensional virtual bundle to \eqref{eq-Kasparov-module-of-bundles}.

\begin{proposition}\label{prop-build-vector-bundle-out-of-k-bundle}
	There is a finite set of sections $\{s_1,s_2,\cdots,s_q\}$ of $\mathcal{H}\times M\to M$ such that the map $\overline{\theta}_\varepsilon: C(M, \mathcal{H}_\varepsilon\oplus \mathbb{C}^q)\to C(M, \mathcal{H})$ given by
	\[
	(u,\lambda_1,\lambda_2,\cdots, \lambda_q)\mapsto  \theta^\mathcal{H}_\varepsilon(u)+\sum_{i=1}^q \lambda_is_i
	\]
	is surjective and whose kernel is a sub-bundle of $\mathcal{H}_\varepsilon\oplus \mathbb{C}^q$. Moreover the $K$-theory element \eqref{eq-Kasparov-module-of-bundles} is equivalent to the virtual bundle $[\operatorname{Ker}(\overline{\theta}_\varepsilon)]-[M\times \mathbb{C}^q]$ in $K_0(C(M))$.
\end{proposition}

\begin{proof}
	For any $m_0\in M$, there is an open neighborhood $U_{m_0}$ where the bundle $\mathcal{H}_\varepsilon$ is trivial. Then the restriction $\theta^{\mathcal{H}}_\varepsilon|_{U_{m_0}}$ can be identified with a map
	\[
	U_{m_0} \xrightarrow{\theta^\mathcal{H}_\varepsilon} \mathcal{F}(\mathcal{H}),
	\]
	here $\mathcal{F}(\mathcal{H})$ denotes the set of Fredholm operators on Hilbert space $\mathcal{H}$. Let $V_0=\operatorname{ker}(\theta_\varepsilon^{\mathcal{H},\ast}(m_0))$, and define $T^0_m: \mathcal{H}_\varepsilon\oplus V_0 \to \mathcal{H}$ to be 
	\[
	T^0_m(u\oplus v) = \theta^\mathcal{H}_\varepsilon u(m)+v.
	\]
	It is surjective at $m=m_0$ and therefore surjective in an open neighborhood $W_0$ of $m_0$. Since $M$ is compact, it can be covered by finitely many such open sets $W_i$ where the maps 
	\[
	T^i_m: \mathcal{H}_\varepsilon\oplus V_i \to \mathcal{H}
	\]
	are surjective. Let $\rho_i$ be the partition of unity associated to the cover $\{W_i\}$. Define
	\[
	\overline{\theta}_\varepsilon: C(M, \mathcal{H}_\varepsilon\bigoplus \oplus_i V_i) \to C(M,\mathcal{H})
	\]
	to be 
	\[
	\overline{\theta}_\varepsilon(u,v_i)(m) = \sum_i \rho_i(m) T^i_m(u,v_i),
	\]
	which is clearly surjective. 
	
	Over the open subset $U_{m_0}$, $\overline{\theta}_\varepsilon$ can be identified with
	$
	U_{m_0}\xrightarrow{\overline{\theta}_\varepsilon} \mathcal{B}(\mathcal{H}\oplus \mathbb{C}^q, \mathcal{H}).
	$
	Its composition with $\mathcal{B}(\mathcal{H}\oplus \mathbb{C}^q, \mathcal{H}) \to \mathcal{B}(\ker(\overline{\theta}_\varepsilon(m_0))^\perp, \mathcal{H})$ is invertible at $m_0$ and therefore invertible on an open neighborhood of $m_0$ where the kernel of $\overline{\theta}_\varepsilon$ is trivial. 
	
	Now, we shall verify the equality
	\[
	\text{The Kasparov module }\eqref{eq-Kasparov-module-of-bundles} = \left(\operatorname{ker}(\overline{\theta}_\varepsilon)\oplus C(M,\mathbb{C}^q), 1, 0\right)
	\]
	in $KK(\mathbb{C}, C(M))$. Indeed, by adding degenerated Kasparov module, we have
	\begin{align*}
	\text{The Kasparov module }\eqref{eq-Kasparov-module-of-bundles} &= \left(C(M, \mathcal{H}_\varepsilon\oplus \mathbb{C}^q\oplus \mathcal{H}\oplus \mathbb{C}^q),1, 
	\begin{bmatrix}
		0 & \theta_\varepsilon^{\mathcal{H},\ast} \\
		\theta_\varepsilon^\mathcal{H} & 0
	\end{bmatrix}
	\right)\\
	&=\left(C(M, \mathcal{H}_\varepsilon\oplus \mathbb{C}^q\oplus \mathcal{H}\oplus \mathbb{C}^q),1, 
	\begin{bmatrix}
		0 & \overline{\theta}_\varepsilon^{\ast} \\
		\overline{\theta}_\varepsilon & 0
	\end{bmatrix}
	\right),
	\end{align*}
here from the first line to the second line is a compact perturbation. Decomposing according to $\overline{\theta}_\varepsilon$, the above equation continues
\begin{align*}
=&\left( \operatorname{ker}(\overline{\theta}_\varepsilon)\oplus \operatorname{ker}(\overline{\theta}_\varepsilon)^\perp\oplus C(M, \mathcal{H})\oplus C(M, \mathbb{C}^q), 1, 
\begin{bmatrix}
	0 & \overline{\theta}_\varepsilon^{\ast} \\
	\overline{\theta}_\varepsilon & 0
\end{bmatrix}
\right)\\
=&\left(\operatorname{ker}(\overline{\theta}_\varepsilon) \oplus C(M, \mathbb{C}^q),1,0\right) \bigoplus
\left(\operatorname{ker}(\overline{\theta}_\varepsilon)^\perp\oplus C(M,\mathcal{H}),1,
\begin{bmatrix}
	0 & \overline{\theta}_\varepsilon^{\ast} \\
	\overline{\theta}_\varepsilon & 0
\end{bmatrix}
\right).
\end{align*}
It suffices to show that the second summand is a degenerate Kasparov module. Indeed, let $U_\varepsilon=\overline{\theta}_\varepsilon(\overline{\theta}_\varepsilon^\ast\overline{\theta}_\varepsilon)^{-1/2}$ be the unitary, and according to the Polar decomposition
\[
\overline{\theta}_\varepsilon = U_\varepsilon(\overline{\theta}_\varepsilon^\ast \overline{\theta}_\varepsilon)^{1/2},
\]
and the fact that $(\overline{\theta}_\varepsilon^\ast \overline{\theta}_\varepsilon)^{1/2}-1$ take value in compact operators,
the operator $\begin{bmatrix}
	0 & \overline{\theta}_\varepsilon^{\ast} \\
	\overline{\theta}_\varepsilon & 0
\end{bmatrix}$ is a compact perturbation of $\begin{bmatrix}
0 & U_\varepsilon^\ast \\
U_\varepsilon & 0
\end{bmatrix}$. This completes the proof.
\end{proof}

\begin{definition}
	Let $E^0_\varepsilon$ be the finite dimensional vector bundle $\ker(\overline{\theta}_\varepsilon)$.
\end{definition}

\subsection{Curvature of $E_\varepsilon^0$}\label{sec-asymptotically-flat}

The Hilbert bundle $\mathcal{H}_\varepsilon$ has connection induced from that of $f^\ast_\varepsilon E$ in the following way: the trivializations \eqref{eq-trivializations-of-pullback} can be viewed as $G/H_\varepsilon$-families of $U_\alpha\times \mathbb{C}^d$. The connection has local form $d+\omega$ on each connected component. So, these connection 1-forms $\omega$'s can be assembled into a single 1-form with value in $G/H_\varepsilon$-families of $M_d(\mathbb{C})$. This can be used as a connection $\nabla^{\mathcal{H}_\varepsilon}$ on $\mathcal{H}_\varepsilon$ whose curvature converges to zero as $\varepsilon\to 0$.

Notice that $U_\varepsilon$ extends to a map from $C(M,\mathcal{H}_\varepsilon\oplus \mathbb{C}^q)$ to $C(M,\mathcal{H})$ whose kernel is $\ker(\overline{\theta}_\varepsilon)$. Let  $U^{-1}_\varepsilon$ be the inverse map from $C(M,\mathcal{H})\to C(M,\mathcal{H}_\varepsilon\oplus\mathbb{C}^q)$ whose range is $\ker(\overline{\theta}_\varepsilon)^\perp$. Then $U_\varepsilon U_\varepsilon^{-1}=\operatorname{id}_{C(M,\mathcal{H})}$. Let $s$ be a section of $\ker(U_\varepsilon)$, then $U_\varepsilon(s)=0$ and
\[
0=\nabla^{\mathcal{H}}(U_\varepsilon(s)) = \nabla^{\mathcal{B}(\mathcal{H}_\varepsilon\oplus \mathbb{C}^q,\mathcal{H})} (U_\varepsilon)(s)+U_\varepsilon(\nabla^{\mathcal{H}_\varepsilon\oplus\mathbb{C}^q}s).
\]
Therefore, the subbundle $E^0_\varepsilon=\ker(\overline{\theta}_\varepsilon)=\ker(U_\varepsilon)$ can be equipped with the following connection
\begin{equation}\label{eq-connection-of-the-finite-dim-bundle}
s\mapsto \nabla^{\mathcal{H}_\varepsilon\oplus\mathbb{C}^q} s+ U_\varepsilon^{-1} (\nabla^{\mathcal{B}(\mathcal{H}_\varepsilon\oplus \mathbb{C}^q,\mathcal{H})} U_\varepsilon) s. 
\end{equation}
It has curvature 
\[
 \nabla^{\mathcal{H}_\varepsilon\oplus\mathbb{C}^q,2}+U_\varepsilon^{-1} (\nabla^{\mathcal{B}(\mathcal{H}_\varepsilon\oplus \mathbb{C}^q,\mathcal{H}),2} U_\varepsilon),
\]
which clearly converges to zero as $\varepsilon\to 0$.

\subsection{Chern Character of $E_\varepsilon^0$}\label{sec-chern-character}

The natural connection on $\mathcal{H}_\varepsilon$ induced from that of $E\to S^n$ have curvature $\Omega^{\mathcal{H}_\varepsilon}$ with value in $G/H_\varepsilon$-families of $M_d(\mathbb{C})$ all but finitely many are zero. Therefore, $\exp(-\Omega^{\mathcal{H}_\varepsilon})$-1 is of trace class. Define the Chern form $\operatorname{ch}(\nabla^{\mathcal{H}_\varepsilon}, \nabla^\mathcal{H})$ to be $\operatorname{tr}(\exp(-\Omega^{\mathcal{H}_\varepsilon})-\exp(-\Omega^\mathcal{H}))$. 

\begin{proposition}
	The cohomology class determined by $\operatorname{ch}(\nabla^{\mathcal{H}_\varepsilon}, \nabla^\mathcal{H})$ is the same as that of $\operatorname{ch}(E^0_\varepsilon-\mathbb{C}^q)$.
\end{proposition}

\begin{proof}
	This follows from the standard transgression argument.
	Let $\nabla^{E^0_\varepsilon}$ be the connection on $E^0_\varepsilon$ defined by the equation \eqref{eq-connection-of-the-finite-dim-bundle}, $\nabla^{E^{0,\perp}_\varepsilon}$ be the trivial connection on $E^{0,\perp}_\varepsilon$ that is pulled back from that of $\mathcal{H}$ by $U_\varepsilon$ and $\nabla^{E^0_\varepsilon\oplus E^{0,\perp}_\varepsilon}$ be the direct sum. 
	
	Let $A_t, 0\leq t\leq 1$ be a family of connections on $\mathcal{H}_\varepsilon\oplus \mathbb{C}^q$ that is defined by $A_t=t\nabla^{\mathcal{H}_\varepsilon\oplus \mathbb{C}^q}+(1-t)\nabla^{E^0_\varepsilon\oplus E^{0,\perp}_\varepsilon}, 0\leq t\leq 1$. Notice that the connection 1-forms of $\nabla^{\mathcal{H}_\varepsilon\oplus \mathbb{C}^q}$, $\nabla^{E^0_\varepsilon\oplus E^{0,\perp}_\varepsilon}$ and the curvatures of $A_t$ are finite rank operators, therefore $\exp(-A_t^2)-1$ are of trace class. Then 
	\begin{align*}
	\frac{d}{dt}\operatorname{tr}\left(\exp(-A_t^2)-\exp(-\Omega^\mathcal{H})\right)&= \operatorname{tr}\left(-\frac{d A_t^2}{dt} \exp(-A_t^2)\right)\\
	&=\operatorname{tr}\left(-[A_t,\frac{d A_t}{dt}\exp(-A_t^2)]\right)\\
	&=d\operatorname{tr}\left(-\frac{d A_t}{dt}\exp(-A_t^2)\right).
	\end{align*}
Here $d A_t/dt$ is a 1-form $\nabla^{\mathcal{H}_\varepsilon\oplus \mathbb{C}^q}-\nabla^{E^0_\varepsilon\oplus E^{0,\perp}_\varepsilon}$ with value in finite rank operators, so the expressions on the left hand side of the equations are of finite rank and, in particular, trace class. As a consequence, 
\[
\operatorname{ch}(\nabla^{\mathcal{H}_\varepsilon}, \nabla^\mathcal{H})-\operatorname{tr}\left(\exp(-\left(\nabla^{E^0_\varepsilon\oplus E^{0,\perp}_\varepsilon}\right)^2)-\exp(-\Omega^\mathcal{H})\right) = d\int^1_0 \operatorname{tr}\left(-\frac{d A_t}{dt}\exp(-A_t^2)\right).
\]
Therefore, the Chern form
$
\operatorname{ch}(\nabla^{\mathcal{H}_\varepsilon}, \nabla^\mathcal{H})
$
determines the same class as \[\operatorname{tr}\left(\exp(-\left(\nabla^{E^0_\varepsilon\oplus E^{0,\perp}_\varepsilon}\right)^2)-\exp(-\Omega^\mathcal{H})\right),\] which is easily computed to equal to $\operatorname{ch}(E^0_\varepsilon-\mathbb{C}^q)$.
\end{proof}

As a consequence of the above proposition,
\[
\langle \widehat{A}(F)\operatorname{ch}(E^0_\varepsilon-\mathbb{C}^q),[M] \rangle = \langle \widehat{A}(F)\operatorname{tr}(\exp(- \Omega^{\mathcal{H}_\varepsilon})-\exp(-\Omega^\mathcal{H})),[M] \rangle.
\]
Over each local trivialization $U_\alpha$, the above integral is equal to 
\[
\int_{U_\alpha} \widehat{A}(F) \operatorname{tr}(\exp(- \Omega^{\mathcal{H}_\varepsilon})-\exp(-\Omega^\mathcal{H})) = \int_{\pi^{-1}_\varepsilon(U_\alpha)} \widehat{A}(\widetilde{F}_\varepsilon) \operatorname{tr}(\exp(- \Omega^{f^\ast_\varepsilon E}) - \exp(-\Omega^\mathcal{H})).
\]
A partition of unity argument proves the equation \eqref{eq-chern-character-integral}.

\bibliography{Refs} 
\bibliographystyle{alpha}

\noindent {\small   Chern Institute of Mathematics, Nankai University, Tianjin 300071, P. R. China. }

\smallskip

\noindent{\small Email: guangxiangsu@nankai.edu.cn}

\medskip

\noindent{\small School of Mathematical Sciences, Tongji University, Shanghai 200092, P. R. China.}

\smallskip

\noindent{\small Email: zelin@tongji.edu.cn}


\end{document}